\documentclass[11pt,twoside, leqno]{article}

\usepackage{amssymb}
\usepackage{amsmath}
\usepackage{mathrsfs}
\usepackage{amsthm}
\usepackage{amsfonts}
\usepackage{enumerate}
\usepackage{esint}
\usepackage{latexsym,amssymb}

\allowdisplaybreaks

\def\ls{\lesssim}

\def\fz{\infty}

\def\r{\right}
\def\lf{\left}

\def\supp{{\mathop\mathrm{\,supp\,}}}

\def\rr{{\mathbb R}}

\def\rn{{{\rr}^n}}
\def\zz{{\mathbb Z}}
\def\nn{{\mathbb N}}

\newcommand{\wz}{\widetilde}

\def\az{\alpha}
\def\lz{\lambda}
\def\blz{\Lambda}

\def\dz {\delta}

\def\fai{\varphi}
\def\gz{{\gamma}}

\def\tz{\theta}

\def\wz{\widetilde}

\def\ls{\lesssim}

\def\boz{\Omega}

\def\pat{\partial}

\def\bmo{{{\mathop\mathrm{bmo}}}}
\def\bbmo{{{\mathop\mathrm{BMO}}}}

\def\hs{\hspace{0.3cm}}

\def\dsum{\displaystyle\sum}
\def\dint{\displaystyle\int}

\def\dsup{\displaystyle\sup}

\def\dlim{\displaystyle\lim}

\newcommand{\I}{{\mathbb I}}

\newtheorem{theorem}{Theorem}[section]

\newtheorem{proposition}[theorem]{Proposition}
\theoremstyle{definition}
\newtheorem{remark}[theorem]{Remark}
\newtheorem{definition}[theorem]{Definition}
\numberwithin{equation}{section}
\def\supp{{\mathop\mathrm{\,supp\,}}}

\def\loc{{\mathop\mathrm{loc\,}}}

\numberwithin{equation}{section}

\begin{document}

\arraycolsep=1pt

\title{\bf\Large
Bilinear Decompositions of Products of Hardy and Lipschitz Spaces
Through Wavelets
\footnotetext {\hspace{-0.35cm}
2010 {\it Mathematics Subject Classification}. Primary: 42B30;
Secondary: 42B35, 46E30, 42C40.
\endgraf {\it Key words and phrases}.
Hardy space, Lipschitz space, product, paraproduct,
renormalization, wavelet, div-curl lemma.
\endgraf
Jun Cao is supported by the National Natural Science Foundation  of China (Grant No. 11501506)
and the Natural Science Foundation of Zhejiang University of Technology (Grant No. 2014XZ011).
Luong Dang Ky is supported by Vietnam National Foundation for Science and Technology
Development (Grant No. 101.02-2014.31).
Dachun Yang is supported by the National
Natural Science Foundation  of China (Grant No. 11571039)
and the Specialized Research Fund for the Doctoral Program of Higher Education
of China (Grant No. 20120003110003).}}
\author{Jun Cao, Luong Dang Ky and Dachun Yang\,\footnote{Corresponding author}}
\date{}
\maketitle

\vspace{-0.6cm}

\begin{center}
\begin{minipage}{13.5cm}
{\small {\bf Abstract} \quad
The aim of this article is to give a complete solution to the problem of
the bilinear decompositions of the products of
some Hardy spaces $H^p(\mathbb{R}^n)$ and their duals in the case when $p<1$ and near to $1$,
via wavelets, paraproducts and the theory of bilinear Calder\'on-Zygmund operators.
Precisely, the authors establish the bilinear decompositions
of the product spaces $H^p(\mathbb{R}^n)\times\dot\Lambda_{\alpha}
(\mathbb{R}^n)$ and $H^p(\mathbb{R}^n)\times\Lambda_{\alpha}(\mathbb{R}^n)$,
where, for all $p\in(\frac{n}{n+1},\,1)$ and $\alpha:=n(\frac{1}{p}-1)$,
$H^p(\mathbb{R}^n)$ denotes the classical real Hardy space, and $\dot\Lambda_{\alpha}$
and $\Lambda_{\alpha}$ denote the homogeneous, respectively, the inhomogeneous
Lipschitz spaces. Sharpness of these two bilinear decompositions are also proved.
As an application, the authors establish some div-curl lemmas at the endpoint case.}
\end{minipage}
\end{center}

\section{Introduction\label{s1}}

\hskip\parindent
The products of functions in the Hardy space
$H^1(\rn)$ and its dual space ${\bbmo}(\rn)$ were first studied by
Bonami et al. in \cite{BIJZ07}. They proved that these products make sense as distributions,
and the product space $H^1(\rn)\times \bbmo(\rn)$ has a linear decomposition of the form
\begin{align*}
H^1(\rn)\times\bbmo(\rn)\subset L^1(\rn)+H_w^\Phi(\rn),
\end{align*}
where $L^1(\rn)$ denotes the usual Lebesgue space and $H_w^\Phi(\rn)$
the weighted Orlicz-Hardy space related to
the weight function $w(x):=\frac{1}{\log(e+|x|)}$ for all $x\in\rn$ and
the Orlicz function
\begin{align*}
\Phi(t):=\frac{t}{\log (e+t)}\ \ \text{for all}\  t\in[0,\,\fz).
\end{align*}

Later, Bonami et al. \cite{BGK12} proved that
the product space $H^1(\rn)\times \bbmo(\rn)$ has a bilinear decomposition of the following form:
\begin{align}\label{eqn 1.s1}
H^1(\rn)\times\bbmo(\rn)\subset L^1(\rn)+H^{\log}(\rn),
\end{align}
where $H^{\log}(\rn)$ denotes the so-called Musielak-Orlicz-Hardy space, introduced in
\cite{Ky14}, related to the Musielak-Orlicz function
\begin{align*}
\log(x,\,t):=\frac{t}{\log(e+|x|)+\log(e+t)} \ \  \text{for all}\ (x,\,t)\in \rn\times[0,\,\fz).
\end{align*}
It is known that
the bilinear decomposition obtained in \cite{BGK12} essentially improves the corresponding result
in \cite{BIJZ07}, due to the fact that the Musielak-Orlicz-Hardy space $H^{\log}(\rn)$ is a proper subspace
of the weighted Orlicz-Hardy space $H_w^\Phi(\rn)$ in \cite{BIJZ07}.  Recall that the study of the
products of functions (or distributions) in the Hardy space and its dual space has applications
in many research areas such as the geometric function theory and the non-linear elasticity
(see \cite{CDM95,BIJZ07} and their references for more details).
We also refer the reader to \cite{Ky13, Ky16} for some
interesting applications of the bilinear decomposition (\ref{eqn 1.s1}) in studying the endpoint estimates
for commutators of Calder\'on-Zygmund operators and $\bbmo(\rn)$ functions.

For the case $p$ less than $1$, Bonami and Feuto \cite{BF10} established
a series of decompositions for the products of the Hardy spaces $H^p(\rn)$ with $p\in(0,\,1)$
and their dual spaces as follows:
\begin{align}\label{eqn 1.w1}
H^p(\rn)\times \blz_{n(\frac{1}{p}-1)}(\rn)\subset L^1(\rn)+H^p(\rn)
\end{align}
and
\begin{align}\label{eqn 1.w2}
H^p(\rn)\times \dot\blz_{n(\frac{1}{p}-1)}(\rn)\subset L^1(\rn)+H_{w_\gz}^p(\rn),
\end{align}
where $H^p(\rn)$ denotes the classical real Hardy space (see \eqref{eqn 3.xx2} below for its
definition), $\dot\blz_{n(\frac{1}{p}-1)}(\rn)$ and $\blz_{n(\frac{1}{p}-1)}(\rn)$
the homogeneous, respectively, the inhomogeneous Lipschitz spaces
(see \eqref{eqn 3.w1} and \eqref{eqn 3.w2} below for their definitions), and $H_{w_\gz}^p(\rn)$
the weighted Hardy space related to the weight function
\begin{align}\label{eqn 1.4x}
w_\gz(x):=\frac{1}{(1+|x|)^\gz} \ \ \ \ \text{for all}\ x\in\rn \ \text{with}\ \gz\in(n(1-p),\,\fz).
\end{align}

All these interesting results
give the decompositions of the products of Hardy spaces and their duals with exponents $p$ less
than $1$. However, it should be pointed out that the decomposition obtained in \eqref{eqn 1.w1}
through \eqref{eqn 1.w2} is not bilinear. Also, as was pointed out in \cite{BF10},
the range space on the right hand side of the decomposition
in \eqref{eqn 1.w2} is not sharp.

Motivated by the aforementioned works, in this article, we give a complete solution to
the problem of bilinear decompositions of the products of Hardy spaces and their
duals in the case when $p<1$ and near to $1$, via wavelets, paraproducts and the
theory of bilinear Calder\'on-Zygmund operators. Precisely, the main result of the present
article is to establish two bilinear
decompositions of the following forms: for all $p\in(\frac{n}{n+1},\,1)$,
\begin{align}\label{eqn 1.xw4}
H^p(\rn)\times \blz_{n(\frac{1}{p}-1)}(\rn)\subset L^1(\rn)+H^p(\rn)
\end{align}
and
\begin{align}\label{eqn 1.w5}
H^p(\rn)\times \dot\blz_{n(\frac{1}{p}-1)}(\rn)\subset L^1(\rn)+H_{w}^p(\rn),
\end{align}
where $H_{w}^p(\rn)$
denotes the weighted Hardy space related to the weight function
\begin{align}\label{1.10x}
w(x):=\frac{1}{(1+|x|)^{n(1-p)}} \ \ \text{for all}\ x\in\rn
\end{align}
(see Theorem \ref{thm 3.8} below for more details).
From their definitions, we immediately deduce that $H_{w}^p(\rn)\subset H_{w_\gz}^p(\rn)$
(see Remark \ref{rem 3.9} below for more details). Thus, our result essentially improves the corresponding
one in \cite{BF10}. Also, we point out that the bilinear decompositions obtained
in \eqref{eqn 1.xw4} and \eqref{eqn 1.w5} are sharp in the sense that the range spaces,
on the right hand sides of the bilinear decompositions, cannot be replaced
by  smaller spaces (see Remark \ref{rem 3.9} below for more details).

As in \cite{BGK12}, the main idea used to obtain the bilinear decompositions \eqref{eqn 1.xw4}
and \eqref{eqn 1.w5} in this article
 bases on the normalization of the products of functions (or distributions)
via wavelets, which was first introduced by Coifman et al.  \cite{CDM95,Do95}. More precisely,
by using the theory of multiresolution analysis (MRA), we decompose the product $f\times g$
of a function (or distribution) $f$ in the Hardy spaces
and an element $g$ in their duals into four parts $\{\Pi_i(f,\,g)\}^4_{i=1}$
(see \eqref{eqn 2.8} through \eqref{eqn 2.11} below for their precise definitions), where,
for each $i\in\{1,\,2,\,3,\,4\}$, the operator $\Pi_i$ is bilinear and can be represented
via a wavelet expansion.
Thus, the idea of renormalization enables us to reduce the problem
to the study of the boundedness of the four bilinear operators
$\{\Pi_i\}_{i=1}^4$ in some suitable function spaces.
By proving that these paraproduct operators are bilinear Calder\'on-Zygmund operators from \cite{GT02-1,GT02-2}
(see Proposition \ref{def 2.4} below), we obtain the boundedness of these bilinear operators
in some suitable function spaces (see Propositions \ref{pro 3.4}, \ref{pro 3.6}
and \ref{pro 3.7} below for more details).
Moreover, to obtain the sharp range space $H_w^p(\rn)$,
we also need to make full use of the wavelet coefficient characterizations
of the Hardy spaces and their duals (see Theorems \ref{thm 3.2} and \ref{thm 3.3} below).
We should point out that it is the use of wavelets that make us to restrict the
range of exponents $p$ to $(\frac{n}{n+1},\,1)$, since we can only obtain $0$-order cancelation moment
condition from the orthogonality of the wavelet basis.

As an application of the main results of this article, we study the div-curl lemma at the endpoint
case. More precisely, let $p\in(\frac{n}{n+1},\,1)$ and $\az:=n(\frac{1}{p}-1)$,
using the bilinear decomposition \eqref{eqn 1.w5},
we are able to prove some div-curl lemmas at the endpoint case $q=\fz$.
We show that, for all  ${\mathbf F}\in H^p(\rn;\,\rn)$ (the vector-valued Hardy space;
see \eqref{eqn 5.1} below for its definition)
with ${\rm curl} \,{\mathbf F}\equiv0$ in the sense of distributions and
${\mathbf G}\in \dot\blz_{\az}(\rn;\,\rn)$ (the vector-valued
homogeneous Lipschitz space) with ${\rm div}\, {\mathbf G}\equiv0$ in the sense of distributions,
it holds true that ${\mathbf F}\cdot {\mathbf G}\in H^p_{w}(\rn)$ with $w$ as in \eqref{1.10x} and
\begin{align}\label{eqn 1.13}
\lf\|{\mathbf F}\cdot {\mathbf G}\r\|_{H_w^p(\rn)}
\le C \|{\mathbf F}\|_{H^p(\rn;\,\rn)}\|{\mathbf G}\|_{\dot\blz_\az(\rn;\,\rn)},
\end{align}
where $C$ is a positive constant, independent of $F$ and $G$.
This result essentially extends the corresponding ones in
\cite{BF10,BGK12}; see also \cite{CLMS93,BFG10} for more related results on
the div-curl lemma.

This article is organized as follows. Section \ref{s2} mainly deals with
the renormalization of the product of two functions (or distributions).
We first, in Section \ref{s2.1}, recall some preliminaries on the theory of multiresolution
analysis (MRA) and the wavelets arising from MRA; then, in Section \ref{s2.2},
we describe the renormalization of the products of functions in $L^2(\rn)$ via wavelets;
in Section \ref{s3}, we prove the bilinear decompositions
\eqref{eqn 1.xw4} and \eqref{eqn 1.w5} (see Theorem \ref{thm 3.8} below);
finally, in Section \ref{s5}, we prove the
div-curl lemma \eqref{eqn 1.13} (see Theorem \ref{thm 5.1} below).

We end this section by making some conventions on the notation.
Throughout the whole article, we always let $\nn:=\{1,2,\ldots\}$
and $\zz_+:=\nn\cup\{0\}$.  We use $C$ to denote a {\it positive
constant} that is independent of the main parameters involved but
whose value may differ from line to line, and $C_{(\az,\ldots)}$ to
denote a {\it positive constant} depending on the indicated parameters $\az,$
$\ldots$. {\it Constants with subscripts}, such as $C_1$, do not
change in different occurrences. If $f\le Cg$, we then write $f\ls
g$ and, if $f\ls g\ls f$, we then write $f\sim g$. For any $\lz\in (0,\fz)$,
$x\in\rr^n$ and $r\in(0,\fz)$, let $$B(x,r):=\{y\in\rr^n:\ |x-y|<r\} \ \
\text{and} \ \lz B:=B(x,\lz r).$$ Also, for any set $E\subset\rn$, $\chi_E$ denotes its
{\it characteristic function}. We let $\mathcal{S}(\rn)$ be the
Schwartz class on $\rn$ and $\mathcal{S}'(\rn)$ the space of all tempered
distributions. The notation $f*g$ always denotes the {\it convolution}
of two functions $f$ and $g$, or a Schwartz function $f$ and a tempered distribution $g$.
For any function $h$, we use $\widehat h$ to denote its {\it Fourier transform}.
Also, for any $x:=(x_1,\,\ldots,\,x_n)\in\rn$ and multi-index $\az:=(\az_1,\,\ldots,\,\az_n)\in \nn^n$, we let
$\pat^\az_x:=(\frac{\pat}{\pat{x_1}})^{\az_1}\cdots(\frac{\pat}{\pat{x_n}})^{\az_n}$
and $x^\az:=x_1^{\az_1}\cdots x_n^{\az_n}$.

\section{Renormalization of Products  via Wavelets}\label{s2}

\hskip\parindent  The main purpose of this section is to
give a renormalization of the pointwise product of any two functions in
$L^2(\rn)$ via wavelets.  This kind of renormalization
constitutes the basis of our method to obtain the bilinear decompositions for
the products of Hardy and Lipschitz or BMO spaces. We first
recall in Section \ref{s2.1} some preliminaries on the homogeneous multiresolution analysis
(MRA) and the wavelets arising from MRA; then, in Section \ref{s2.2},
we renormalize the products of functions in $L^2(\rn)$ into four bilinear operators
through wavelets.

\subsection{MRA and Wavelets}\label{s2.1}

\hskip\parindent
We begin this subsection by recalling the definition of the (homogeneous)
multiresolution analysis on $\rr$;
see, for example, \cite{Me92,HW96} for more details.

\begin{definition}\label{def homoMRA}
Let $\{V_j\}_{j\in\zz}$ be an increasing sequence of closed subspaces in $L^2(\rr)$.
Then $\{V_j\}_{j\in\zz}$ is called a {\it multiresolution analysis} (for short, MRA)
\emph{on $\rr$} if it has the following
properties:

\begin{itemize}
\item [{\rm(i)}] $\bigcap_{j\in\zz} V_j=\{\tz\}$ and $\overline{\bigcup_{j\in\zz} V_j}=L^2(\rr)$,
where $\tz$ is the zero element of $L^2(\rr)$;

\item [{\rm(ii)}] for any $j\in\zz$ and $f\in L^2(\rr)$, $f(\cdot)\in V_j$ if and only if $f(2\cdot)\in V_{j+1}$;

\item [{\rm(iii)}] for any $f\in L^2(\rr)$ and $k\in\zz$, $f(\cdot)\in V_0$ if and only if $f(\cdot-k)\in V_{0}$;

\item [{\rm(iv)}] there exists a function $\phi\in L^2(\rr)$ (called a {\it scaling function} or {\it father wavelet})
such that
$\{\phi_k(\cdot)\}_{k\in\zz}:=\{\phi(\cdot-k)\}_{k\in\zz}$ is an orthonormal basis of $V_0$.
\end{itemize}
\end{definition}

For all $j\in\zz$, let $W_j:=V_{j+1}\ominus V_{j}$ be the {\it orthogonal complement} of
$V_{j}$ in $V_{j+1}$. Then it is easy to see that
$V_{j+1}=\bigoplus_{i=-\fz}^j W_i$ and
$L^2(\rr)=\bigoplus_{i=-\fz}^{\fz} W_i$,
here and hereafter, $\bigoplus$ denotes the {\it orthogonal direct sum} in $L^2(\rr)$.

By Definition \ref{def homoMRA}, we know that, for all $j\in\zz$,
$\{2^{j/2}\phi(2^j\cdot-k)\}_{k\in\zz}$ is an orthogonal basis of $V_j$.
Also, from the classical theory of MRA (see, for example, \cite{Da88,Me92,HW96}),
it follows that there exists a function
$\psi\in L^2(\rr)$, called a {\it wavelet} or {\it mother wavelet}, such that, for all $j\in\zz$,
$\{2^{j/2}\psi(2^j\cdot-k)\}_{k\in\zz}$ is an orthogonal basis of $W_j$.
Moreover, for any $N \in\nn$, it is known that one can construct the father and the mother wavelets
$\phi,\,\psi\in C^N(\rr)$ (the set of all functions with continuous derivatives up to order $N$)
with compact supports such
that $\widehat{\phi}(0) = (2\pi)^{-1/2}$, where $\widehat{\phi}$ denotes the {\it Fourier transform} of $\phi$,
and, for all $l\in\{0,\,\ldots,\,N\}$,
\begin{align}\label{eqn 2.2}
\dint_\rr x^l \psi(x)\,dx=0
\end{align}
(see, for example, \cite{Da88}).
Recall that there does not exist a wavelet basis in $L^2(\rr)$ whose elements are
both infinitely differentiable and have compact supports (see, for example, \cite[Theorem 3.8]{HW96}).

Following \cite{BGK12}, throughout the whole article, we always assume that
\begin{align}\label{eqn 2.3}
\supp \phi,\ \supp \psi\subset 1/2+m(-1/2,\,1/2),
\end{align}
where $1/2+m(-1/2,\,1/2)$ denotes the interval obtained from $(0,\,1)$ via a dilation by $m$ centered
at $1/2$, namely, $x\in 1/2+m(-1/2,\,1/2)$ if and only if $|x-1/2|<{m}/{2}$. Here
$m$ is a positive constant that is independent of the main parameters considered in the whole article.

The extension of the above considerations from one dimension to any
$n$-dimension follows from a standard procedure of tensor products. More precisely,
let $\{\wz V_j\}_{j\in\zz}$ be a multiresolution analysis on $\rr$. For all $j\in\zz$ and $n\in\nn$, let
\begin{align}\label{eqn 2.4}
V_j:=\overbrace{\wz {V}_j \otimes \cdots \otimes \wz {V}_j}^{n {\textrm{ times}}}
\end{align}
be the {\it $n$-fold tensor product} of $\wz {V}_j$ and
$W_j:=V_{j+1}\ominus V_{j}$ be the {\it orthogonal complement} of
$V_{j}$ in $V_{j+1}$. Then it is easy to see that
\begin{align}\label{eqn 2.5}
V_{j+1}=\bigoplus_{i=-\fz}^j W_i \ \
\text{and} \ \
L^2(\rn)=\bigoplus_{i=-\fz}^{\fz} W_i.
\end{align}
Moreover, let $E:=\{0,\,1\}^n\setminus\{(\overbrace{0,\,\ldots,\,0}^{n {\textrm{ times}}})\}$. For all $\lz:=(\lz_1,\ldots,\,\lz_n)\in E$
and $x:=(x_1,\,\ldots,\,x_n)\in \rn$, let
\begin{align}\label{eqn 2.q6}
\psi^\lz(x):=\phi^{\lz_1}(x_1)\cdots\phi^{\lz_n}(x_n),
\end{align}
where, for all $i\in\{1,\,\ldots,\,n\}$,
\begin{align*}
\phi^{\lz_i}(x_i):=\begin{cases}
\phi(x_i) \ \ & \text{when}\ \lz_i=0,
\\
\psi(x_i) \ \ & \text{when}\ \lz_i=1.
\end{cases}
\end{align*}
For all $j\in\zz$ and $k:=\{k_1,\,\ldots,\,k_n\}\in\zz^n$, let $$I_{j,\,k}:=\{x\in\rn:\ k_i\le 2^j x_i<k_i+1 \
\text{for all}\ i\in\{1,\,\ldots,\,n\}\}$$ be the {\it dyadic cube} with the lower left-corner $x_{j,\,k}:=
2^{-j}k$. For all $x:=(x_1,\,\ldots,\,x_n)\in\rn$, let
\begin{align*}
\phi_{I_{0,\,0}}(x):=\phi(x_1)\cdots\phi(x_n).
\end{align*}
From Definition \ref{def homoMRA} and \eqref{eqn 2.5}, we deduce that, for all $j\in\zz$,
$$\lf\{\phi_{I_{j,\,k}}(\cdot)\r\}_{k\in\zz^n}:=\lf\{2^{jn/2}\phi_{I_{0,\,0}}(2^j\cdot-k)\r\}_{k\in\zz^n}$$ and
$$\lf\{\psi^\lz_{I_{j,\,k}}(\cdot)\r\}_{\lz\in E,\,k\in\zz^n}
:=\lf\{2^{jn/2}\psi^\lz(2^j\cdot-k)\r\}_{\lz\in E,\, k\in\zz^n}$$
form orthogonal bases of $V_j$, respectively, $W_j$. {\it In what follows,
to simplify the notation,
we often write $I$ instead $I_{j,\,k}$ when there exists no confusion}. By \eqref{eqn 2.3},
we conclude that, for all dyadic cubes $I$,
\begin{align}\label{eqn 2.6}
\supp \phi_I,\ \supp \psi_I\subset m I,
\end{align}
where $mI$ denotes the $m$ dilation of $I$ with the same center as $I$.
Thus, for any $f\in L^2(\rn)$, we have the following wavelet expansion
$$f=\sum_{I\in \mathcal{D}}\sum_{\lz\in E} \langle f,\,\psi_I^\lz \rangle \psi_I^\lz$$
holds true in $L^2(\rn)$, here and hereafter, $\mathcal{D}$ always denotes the set of all the
{\it dyadic cubes} in $\rn$ and, unless otherwise stated,
$\langle \cdot,\,\cdot \rangle$ the inner product in $L^2(\rn)$.

\subsection{Renormalization of Pointwise Products in $L^2(\rn)\times L^2(\rn)$}\label{s2.2}

\hskip\parindent  Let $f$ and $g$  be two functions in $L^2(\rn)$. In this subsection,
we renormalize the product $fg$ by using the wavelet basis in Section \ref{s2.1}.
Before continuing, we first point out that this method of renormalization was first introduced
by Coifman et al. \cite{CDM95} and Dobyinsky \cite{Do95}.

For all $j\in\zz$, let $V_j$ and $W_j$ be as in Section \ref{s2.1}, and
$P_j$ and $Q_j$ be the {\it orthogonal projections} of $L^2(\rn)$ onto $V_j$, respectively,
onto $W_j$. Dobyinsky \cite{Do95} proved that, for all $f$, $g\in L^2(\rn)$,
\begin{align}\label{eqn 2.7}
fg=\dsum_{j\in\zz} (P_j f)(Q_j g)+\dsum_{j\in\zz} (Q_j f)(P_j g)+\dsum_{j\in\zz} (Q_j f)(Q_j g)
\end{align}
holds true in $L^1(\rn)$.
Based on \eqref{eqn 2.7} and following \cite{Do95,BGK12}, we introduce 
four bilinear operators via the wavelet basis as follows.
For all $f$, $g\in L^2(\rn)$, define
\begin{align}\label{eqn 2.8}
\Pi_1(f,\,g):=\dsum_{\substack{I,\,I'\in \mathcal{D}\\
|I|=|I'|}} \dsum_{\lz\in E} \langle f,\,\phi_I \rangle
\langle g,\,\psi_{I'}^\lz \rangle \phi_I \psi_{I'}^\lz,
\end{align}
\begin{align}\label{eqn 2.9}
\Pi_2(f,\,g):=\dsum_{\substack{I,\,I'\in \mathcal{D}\\
|I|=|I'|}} \dsum_{\lz\in E} \langle f,\,\psi_I^\lz \rangle
\langle g,\,\phi_{I'} \rangle \psi_I^\lz \phi_{I'},
\end{align}
\begin{align}\label{eqn 2.10}
\Pi_3(f,\,g):=\dsum_{\substack{I,\,I'\in \mathcal{D}\\
|I|=|I'|}} \dsum_{\substack{\lz,\,\lz'\in E\\
(I,\,\lz)\ne (I',\,\lz')}} \langle f,\,\psi_I^\lz \rangle
\langle g,\,\psi_{I'}^{\lz'} \rangle \psi_I^\lz \psi_{I'}^{\lz'}
\end{align}
and
\begin{align}\label{eqn 2.11}
\Pi_4(f,\,g):=\dsum_{I\in \mathcal{D}} \dsum_{\lz\in E} \langle f,\,\psi_I^\lz \rangle
\langle g,\,\psi_{I}^\lz \rangle \lf(\psi_{I}^\lz\r)^2.
\end{align}
We point out that, although the bilinear operators   are only defined in $L^2(\rn)\times L^2(\rn)$, they can be extended to
the products of more general function spaces in the case even when $f$ and $g$
are distributions and have wavelet
expansions (see, for example, \cite{Me92,HW96} for the Hardy spaces
or \cite{Tr06,Tr08} for more general scales of Besov and Triebel-Lizorkin spaces).
In this case, the symbol $\langle \cdot,\,\cdot \rangle$ may denote the dual pair between
distribution and the associated wavelet function.

From \eqref{eqn 2.7} and the facts that $\lf\{\phi_{I_{j,\,k}}\r\}_{k\in\zz^n}$ and
$\{\psi^\lz_{I_{j,\,k}}\}_{\lz\in E,\, k\in\zz^n}$  constitute orthogonal
bases of $V_j$, respectively, $W_j$, it follows that
\begin{align*}
fg=\sum^4_{i=1}\Pi_i(f,\,g)
\end{align*}
holds true in $L^1(\rn)$. As is in \cite{Do95,BGK12},
we always let $T:=\sum^3_{i=1}\Pi_i$ and $S:=\Pi_4$. It will
be seen that both the operators $S$ and $T$ inherit some properties of the factors
$f$ and $g$.
In particular, the operator $T$, which preserves the cancelation properties of the Hardy spaces,
is usually called the {\it renormalization of the product} $fg$.

The remainder of this section is devoted to the study of
some basic bounded properties of the bilinear operators
$\{\Pi_i\}_{i=1}^4$. To this end, we  first recall the definition of the Hardy space
$H^p(\rn)$ for $p\in(0,\,\fz)$ from \cite{FS72}.
For any $m\in\nn$, let
\begin{align*}
\mathcal{S}_m(\rn):=\lf\{\fai\in \mathcal{S}(\rn):\
\dsup_{\substack{x\in\rn\\ |\az|\le m+1}} \lf(1+|x|\r)^{(m+2)(n+1)}\lf|\pat_x^\az \fai (x)\r|\le 1\r\}.
\end{align*}
For any $m\in\nn$, $f\in \mathcal{S}'(\rn)$ and $x\in\rn$, define
the {\it non-tangential grand maximal function} $f_m^*(x)$ of $f$ by setting
\begin{align*}
f_m^*(x):=\dsup_{\fai\in \mathcal{S}_m(\rn)}\dsup_{|y-x|<t,\, t\in (0,\fz)}\lf|f*\fai_t(y)\r|,
\end{align*}
here and hereafter, for all $t\in(0,\,\fz)$, $\fai_t(\cdot):=\frac{1}{t^n}\fai(\frac{\cdot}{t})$
and we always remove the subscript $m$ whenever
$m>\lfloor n(\frac{1}{p}-1)\rfloor$. Recall that, as usual, the \emph{symbol $\lfloor s \rfloor$
for any $s\in\rr$ denotes the largest integer smaller than or equal to $s$}.

Then, for all $p\in(0,\,\fz)$, the {\it Hardy space} $H^p(\rn)$ is defined to be the set
\begin{align}\label{eqn 3.xx2}
H^p(\rn):=\{f\in \mathcal{S}'(\rn):\ \|f\|_{H^p(\rn)}:=\|f^*\|_{L^p(\rn)}<\fz\}.
\end{align}
For more properties of the Hardy spaces and their various characterizations, we refer
the reader to \cite{FS72,St93}.

The following theorem, on the boundedness of the operators $S$ and $T$,
was first proved by Dobyinsky \cite{Do95}.

\begin{theorem}[\cite{Do95}]\label{thm 2.2}
{\rm (i)} The operator $S$ is a bilinear operator bounded from $L^2(\rn)\times L^2(\rn)$
to $L^1(\rn)$.

{\rm (ii)} The operator $T$ is a bilinear operator bounded from $L^2(\rn)\times L^2(\rn)$
to $H^1(\rn)$.
\end{theorem}

\begin{remark}\label{rem 2.3}
From Theorem \ref{thm 2.2}, we deduce that, for any $f$, $g\in L^2(\rn)$,
$$fg=S(f,\,g)+T(f,\,g)\in L^1(\rn)+H^1(\rn).$$
This implies the following bilinear decomposition of the product space $L^2(\rn)\times L^2(\rn)$,
$$L^2(\rn)\times L^2(\rn)\subset L^1(\rn)+H^1(\rn).$$
\end{remark}

Now we show that the four bilinear operators
$\{\Pi_i\}_{i=1}^4$, defined as in \eqref{eqn 2.8} through \eqref{eqn 2.11}, are
bilinear Calder\'on-Zygmund operators. This fact is needed when we discuss the
boundedness properties of these operators later in this article.
Recall that, in \cite{GT02-1,GT02-2}, a bilinear operator $T$ is called a {\it bilinear Calder\'on-Zygmund operator} if it satisfies the following
two conditions:

\begin{enumerate}
\item[(i)] there exist  $p_0,\,q_0\in(1,\,\fz)$ and $r_0\in[1,\,\fz)$
satisfying $\frac{1}{r_0}=\frac{1}{p_0}+\frac{1}{q_0}$ such that $T$ can be extended to a bounded
bilinear operator from $L^{p_0}(\rn)\times L^{q_0}(\rn)$ to $L^{r_0}(\rn)$;

\item[(ii)] the associated kernel function $K$ of $T$ satisfies the following
size and the following regularity conditions: there exists a positive constant $C$ such that,
for all $(x,\,y,\,z)\in (\rn)^3\setminus \Omega$,
\begin{align*}
\lf|K(x,\,y,\,z)\r|\le C\frac 1{(|x-y|+|x-z|+|y-z|)^{2n}}
\end{align*}
and there exist positive constants $\epsilon\in(0,\,1]$ and $C$ such that, for all $(x,\,y,\,z)$,
$(x',\,y,\,z)\in(\rn)^3\setminus\boz$ satisfying $|x-x'|\le
\frac{1}{2} \max\{|x-y|,\,|y-z|\}$,
\begin{align}\label{eqn 2.14}
\lf|K(x,\,y,\,z)-K(x',\,y,\,z)\r|\le C\frac{|x-x'|^\epsilon}{(|x-y|+|x-z|+|y-z|)^{2n+\epsilon}},
\end{align}
where $(\rn)^3:=\rn\times \rn\times \rn$ and
$\Omega:=\{(x,\,y,\,z)\in (\rn)^3:\ x=y=z\}$ denotes the {\it diagonal set} in $(\rn)^3$.
Moreover, for all $(x,\,y,\,z)\in (\rn)^3\setminus \Omega$, let $K_1(x,\,y,\,z):=K(y,\,x,\,z)$
and $K_2(x,\,y,\,z):=K(z,\,y,\,x)$. Then $K_1$ and $K_2$ also satisfy \eqref{eqn 2.14}.
\end{enumerate}

\begin{proposition}\label{def 2.4}
The bilinear operators $\{\Pi_i\}_{i=1}^4$, defined as in \eqref{eqn 2.8} through \eqref{eqn 2.11},
are bilinear Calder\'on-Zygmund operators with $\epsilon=1$.
\end{proposition}

\begin{proof}
Without loss of generality, we may only consider the bilinear operator $\Pi_1$.
The proofs of the bilinear operators $\{\Pi_i\}_{i=2}^4$ are similar, the
details being omitted. By \eqref{eqn 2.8} and \eqref{eqn 2.6}, we first write
\begin{align}\label{eqnn 2.17}
\Pi_1(f,\,g)=&\dsum_{k'\in(-m,\,m]^n}\dsum_{\lz\in E}
\lf[\dsum_{I\in \mathcal{D}} |I|^{-1/2}\langle f,\,\phi_I \rangle
\langle g,\,\psi_{I+\ell_Ik'}^\lz \rangle |I|^{1/2}\phi_I \psi_{I+\ell_Ik'}^\lz\r]\\
=&\nonumber\hspace{-0.1cm}:\dsum_{k'\in(-m,\,m]^n}\dsum_{\lz\in E} \lf[\wz{\Pi}_{1,\,k',\,\lz}(f,\,g)\r],
\end{align}
where, for each $I:=I_{j,\,k}\in \mathcal{D}$, $\ell_I:=2^{-j}$ denotes its side length,
and we assume that the father and the mother wavelets $\phi$,  $\psi\in C^N(\rr)$ with
some fixed $N\in(1,\,\fz)$.
This, together with \eqref{eqn 2.6}, immediately implies that,
for any  $M\in \nn\cap (n,\,\fz)$, multi-indices $\gz$ satisfying $|\gz|\le N$ and $x\in\rn$,
\begin{align}\label{eqn 2.15}
\lf|\partial^\gz (|I|^{1/2}\phi_I \psi_{I+\ell_Ik'}^\lz)(x)\r| \le C_{(\gz,\,N,\,M)}\frac{2^{jn/2}2^{j|\gz|}}{(1+2^{j}|x-2^{-j}k|)^M},
\end{align}
where $C_{(\gz,\,N,\,M)}$ is a positive constant depending on $\gz,\,N$ and $M$, but independent of
$j$, $k$, $k'$ and $x$.
Moreover, from  $\psi_{I+\ell_Ik'}^\lz\in W_j$ and $\phi_I\in V_j$, we deduce that
\begin{align}\label{eqn 2.16xxx}
\dint_\rn  (|I|^{1/2}\phi_I \psi_{I+\ell_Ik'}^\lz)(x)\,dx=0.
\end{align}
This, combined with \eqref{eqn 2.15}, shows $\{|I|^{1/2}\phi_I \psi_{I+\ell_Ik'}^\lz\}_{I\in\mathcal{D}}$
is a family of smooth $(N,\,M,\,0)$-molecules (see \cite[p.\,56, (3.3) through (3.6) ]{FJ90}
or \cite[(3) and (5)]{BMNT10} for their definitions).
By using a similar calculation, we also find that $\{\psi_{I+\ell_Ik'}^\lz\}_{I\in\mathcal{D}}$
and $\{\phi_I \psi_{I+\ell_Ik'}^\lz\}_{I\in\mathcal{D}}$ satisfy both \eqref{eqn 2.15} and \eqref{eqn 2.16xxx},
and hence are families of smooth $(N,\,M,\,0)$-molecules.
Thus, we conclude that  $\wz{\Pi}_{1,\,k',\,\lz}(f,\,g)$ is a typical {\it dyadic paraproduct} of the following form:
\begin{align*}
\Pi (f,\,g):=\dsum_{I\in\mathcal{D}} |I|^{-1/2} \langle f,\,\phi_I^1 \rangle
\langle g,\,\phi_I^2 \rangle \phi_I^3,
\end{align*}
where $\{\phi_I^j\}_{I\in\mathcal{D}}$, for any $j\in\{1,\,2,\,3\}$, is a family of smooth molecules.
By \cite[Theorem 4.1]{BMNT10},
we find out that, for any $k'\in(-m,\,m]^n$ and $\lz\in E$,
the bilinear operator $\wz{\Pi}_{1,\,k',\,\lz}$
is a bilinear Calder\'on-Zygmund operator with $\epsilon=1$,
which, together with \eqref{eqnn 2.17}, implies that the same conclusion
holds true for $\Pi_1$. This finishes the proof of
Proposition \ref{def 2.4}.
\end{proof}

\begin{remark}\label{rem2.5}
Using Proposition \ref{def 2.4} and
the boundedness properties of bilinear Calder\'on-Zygmund operators (see, for example,
\cite[Theorem 4.2]{BMNT10}),
we immediately know that, for any $i\in\{1,\,2,\,3,\,4\}$ and $p,\,q\in(1,\,\fz)$, $\Pi_i$
is bounded from the product space $L^p(\rn)\times L^q(\rn)$
to $L^r(\rn)$ with $\frac{1}{r}:=\frac{1}{p}+\frac{1}{q}$.
\end{remark}

\section{Products of Functions in Hardy and Lipschitz Spaces}\label{s3}

\hskip\parindent   In this section, we study the bilinear decompositions of the
products of functions (or distributions) in Hardy and Lipschitz spaces.

For all $\az\in(0,\,1]$ and $f$ being continuous, let
\begin{align*}
\|f\|_{\dot\blz_\az(\rn)}:=\dsup_{\substack{x,\,y\in\rn\\ x\ne y}} \lf\{\frac{|f(x)-f(y)|}{|x-y|^\az}\r\}
\ \ \text{and}\ \
\|f\|_{\blz_\az(\rn)}:=\|f\|_{\dot\blz_\az(\rn)}+\|f\|_{L^\fz(\rn)}.
\end{align*}
Then the {\it homogeneous} and the {\it inhomogeneous Lipschitz spaces} are defined as follows:
\begin{align}\label{eqn 3.w1}
\dot \blz_\az(\rn):=\{f \ \text{is continuous in}\ \rn: \ \|f\|_{\dot\blz_\az(\rn)}<\fz\},
\end{align}
respectively, 
\begin{align}\label{eqn 3.w2}
\blz_\az(\rn):=\{f \ \text{is continuous in}\ \rn: \ \|f\|_{\blz_\az(\rn)}<\fz\}.
\end{align}
By their definitions, it is easy to see that $\blz_\az(\rn)=[\dot \blz_\az(\rn)\cap L^\fz(\rn)]$.

It is well known that the homogeneous Lipschitz space $\dot \blz_\az(\rn)$ can be
characterized by the mean oscillation as follows.

\begin{theorem}[\cite{Me64,JTW83}]\label{thm 3.1}
Let $\az\in[0,\,1]$ and $q\in[1,\,\fz)$. For all $f\in L_{\rm loc}^1(\rn)$, define
\begin{align*}
\|f\|_{\bbmo_{\az,\,q}(\rn)}:=\dsup_{\substack{x\in\rn\\
r\in(0,\,\fz)}}\lf\{\frac{1}{|B(x,\,r)|^{1+\frac{\az}n}}{\dint_{B(x,\,r)}\lf|f(y)-f_{B(x,\,r)}\r|^q\,dy}\r\}^{\frac{1}{q}}
\end{align*}
and $\bbmo_{\az,\,q}(\rn):=\{f\in L_{\rm loc}^1(\rn):\ \|f\|_{\bbmo_{\az,\,q}(\rn)}<\fz\},$
where, for any ball $B$, $f_B:=\frac{1}{|B|}\int_B f(x)\,dx$ denotes the mean of $f$
over $B$. Then, for all $\az\in(0,\,1]$ and $q\in[1,\,\fz)$,
$\dot \blz_\az(\rn)=\bbmo_{\az,\,q}(\rn)$ with equivalent norms.
\end{theorem}
Recall that, if $\az\equiv0$, then the space $\bbmo_{\az,\,q}(\rn)$ is just the well known space
$\bbmo(\rn)$,
which is identified as the dual space of the Hardy space $H^1(\rn)$. Also, since the space
$\bbmo_{\az,\,q}(\rn)$ is invariant under the change of $q$, we usually remove the subscript
$q$ when there exists no confusion.

The following result characterizes the homogeneous Lipschitz space $\dot \blz_\az(\rn)$
in terms of the wavelet coefficients, which plays an important role in what follows.

\begin{theorem}[\cite{AB97}]\label{thm 3.2}
Let $\az\in[0,\,1]$, $f\in\bbmo_\az(\rn)$ and $\{\psi_I^{\lz}\}_{I\in\mathcal{D},\,\lz\in E}$
be a family of mother wavelets.
Then the wavelet coefficients $\{s_{I,\,\lz}\}_{I\in\mathcal{D},\,\lz\in {\rm E}}
:=\{\langle f,\,\psi^\lz_I\rangle\}_{I\in\mathcal{D},\,\lz\in {\rm E}}$ of $f$
satisfy the following Carleson type condition, namely, there exists a positive constant $C$,
independent of $f$, such that
\begin{align*}
\lf\|\{s_{I,\,\lz}\}_{I\in\mathcal{D},\,\lz\in {E}}\r\|_{\mathcal{C}_\az(\rn)}:=\dsup_{I\in \mathcal{D}}
\lf\{\frac{1}{|I|^{\frac{2\az}{n}+1}}\dsum_{\lz\in E}\dsum_{\substack{J\in\mathcal{D}
\\J\subset I}}\lf|s_{J,\,\lz}\r|^2\r\}^{\frac{1}{2}}\le C\|f\|_{\bbmo_\az(\rn)}.
\end{align*}
\end{theorem}

For the wavelet characterization of the Hardy space $H^p(\rn)$ defined
as in \eqref{eqn 3.xx2} with $p\in(0,\,\fz)$,
let $\psi\in C^N(\rr)$ be a mother wavelet satisfying \eqref{eqn 2.2} and
\eqref{eqn 2.3} with $N$ sufficiently large. Let $\{\psi_I^\lz\}_{I\in\mathcal{D},\,\lz\in {E}}$
be a family of the  bases generated by $\psi$. Recall that, for all $p\in(0,\,\fz)$,
$H^p(\rn)$ coincides with the homogeneous Triebel-Lizorkin space
$\dot F^{0}_{p,\,2}(\rn)$ with equivalent quasi-norms (see, for example, \cite[p.\,238, Definition 2]{Tr83} for the
definition of the homogeneous Triebel-Lizorkin space $\dot F^{0}_{p,\,2}(\rn)$ and
\cite[p.\,244, Theorem]{Tr83} for the statement of this fact).
Thus, by the wavelet coefficient characterization of
$\dot F^{0}_{p,\,2}(\rn)$ (see \cite[Theorem 2.2]{FJ90}), we know that
there exists a positive constant $\wz C$ such that, for all $f\in H^p(\rn)$,
\begin{align} \label{3.3xx}
\|f\|_{H^p(\rn)}&=\wz C \lf\|\lf\{ \dsum_{\substack{I\in\mathcal{D}\\
\lz\in E}} \lf[\lf|\langle f,\,\psi_I^\lz \rangle\r| |I|^{-\frac{1}2}
\chi_I(\cdot)\r]^2\r\}^{\frac{1}{2}} \r\|_{L^p(\rn)}\hspace{-0.2cm}\\
&\nonumber=:\wz C\lf\|\{\langle f,\,\psi_I^\lz \rangle\}_{I\in\mathcal{D},\,\lz\in E}\r\|_{\dot
f^0_{p,\,2}(\rn)},
\end{align}
where $\dot f^0_{p,\,2}(\rn)$ denotes the corresponding {\it homogeneous Triebel-Lizorkin sequence
space}  (see \cite[p.\,38]{FJ90} for its definition).

Moreover, if the element
of the Hardy space $H^p(\rn)$ has a finite wavelet expansion,
we can obtain some further properties on the atomic decomposition of $H^p(\rn)$.
Recall that, for all $p\in(0,\,1]$,
a function $a\in L^2(\rn)$ is called an $H^p(\rn)${\it-atom} if there exists
a cube $I$ such that $\supp a \subset I$, $\|a\|_{L^2(\rn)}\le |I|^{\frac{1}{2}-\frac{1}{p}}$
and, for any multi-index $\az:=(\az_1,\ldots,\az_n)\in(\zz_+)^n$ satisfying $|\az|\le\lfloor n(\frac{1}{p}-1)\rfloor$,
\begin{align*}
\dint_\rn x^\az a(x)\,dx=0,
\end{align*}
where, for any $x:=(x_1,\ldots, x_n)\in\rn$, $x^\az:=x_1^{\az_1}\cdots x_n^{\az_n}$.

The following theorem can be proved by a way similar
to that used in the proof of
\cite[Theorem 5.12]{HW96}, the details being omitted.

\begin{theorem}\label{thm 3.3}
Let $p\in(0,\,1]$ and $f\in H^p(\rn)$ have a finite wavelet expansion, namely,
\begin{align}\label{eqn 3.7}
f=\dsum_{I\in\mathcal{D}}\dsum_{\lz\in E} \langle f,\,\psi_I^\lz \rangle \psi_I^\lz,
\end{align}
where the coefficients $\langle f,\,\psi_I^\lz \rangle\ne 0$ only for a
finite number of $(I,\,\lz)\in \mathcal{D}\times E$.
Then $f$ has a finite atomic decomposition satisfying
$f=\sum_{l=1}^L \mu_l a_l$, with $L\in\nn$, and there exists a positive constant $C$,
independent of $\{\mu_l\}_{l=1}^L$, $\{a_l\}_{l=1}^L$
and $f$, such that 
$$\lf\{\sum_{l=1}^L |\mu_l|^p\r\}^{\frac{1}{p}}\le C \|f\|_{H^p(\rn)},$$
where, for all $l\in\{1,\,\ldots,\,L\}$, $a_l$ is an $H^p(\rn)$-atom
associated with some dyadic cube $R_l$, which can be written into the following form
\begin{align}\label{eqn 3.x4}
a_l=\dsum_{\substack{I\in\mathcal{D}\\ I\subset R_l}}\dsum_{\lz\in E} c_{I,\,\lz,\,l}\psi_I^\lz
\end{align}
with $\{c_{I,\,\lz,\,l}\}_{I\subset R_l,\,\lz\in E,\,l\in\{1,\,\ldots,\,L\}}$ being positive constants independent of
$\{a_l\}_{l=1}^L$. Moreover, for each $l\in\{1,\,\ldots,\,L\}$,
the wavelet expansion of $a_l$ in \eqref{eqn 3.x4} is also finite which is
extracted from that of $f$ in \eqref{eqn 3.7}.
\end{theorem}

Now, we consider the product of the Hardy and the Lipschitz spaces. Let
$f\in H^p(\rn)$ and $g\in \dot \blz_\az(\rn)$ with $\az=n(\frac{1}{p}-1)$.
Observe that $f$ may be a distribution. Thus, we cannot define the product of $f$ and $g$
directly by pointwise multiplication. However, as was pointed out in \cite{BIJZ07,BF10}, we can
define the product as the distribution in the following way. For any $p\in(0,\,1)$ and $\az\in(0,\,1)$
with $\az=n(\frac{1}{p}-1)$,
let $f\in H^p(\rn)$ and $g\in \dot \blz_\az(\rn)$. The {\it product} $f\times g$ is defined by setting,
for any $\phi\in \mathcal{S}(\rn)$,
\begin{align}\label{eqn 3.3}
\langle f\times g,\,\phi \rangle:=\langle \phi g,\,f \rangle,
\end{align}
where the last bracket denotes the dual pair between
$\dot \blz_{n(\frac{1}{p}-1)}(\rn)$ and $H^p(\rn)$.

Recall that Nakai and Yabuta \cite[Theorem 1]{NY85} proved that every
$\phi \in \mathcal{S}(\rn)$ is a pointwise multiplier of the homogeneous Lipschitz space
$\dot \blz_{\az}(\rn)$ for all $\az\in(0,\,1]$, namely, for all $g\in \dot \blz_{\az}(\rn)$,
$\phi g\in \dot \blz_{\az}(\rn)$. This implies that the definition in \eqref{eqn 3.3} is meaningful.

If $f\in H^p(\rn)$ and $g\in \blz_\az(\rn)$, where $p\in(0,\,1)$ and $\az=n(\frac{1}{p}-1)$,
then, similar to \eqref{eqn 3.3},  we can define
the product $f\times g$ by setting, for any $\phi\in \mathcal{S}(\rn)$,
\begin{align}\label{eqn 3.4}
\langle f\times g,\,\phi \rangle:=\langle \phi g,\,f \rangle,
\end{align}
where the last bracket denotes also the dual pair between
$\dot \blz_{n(\frac{1}{p}-1)}(\rn)$ and $H^p(\rn)$.  This definition
is well defined because of the facts that $\blz_{n(\frac{1}{p}-1)}(\rn)\subset
\dot \blz_{n(\frac{1}{p}-1)}(\rn)$ and that every
$\phi \in \mathcal{S}(\rn)$ is also a pointwise multiplier of the inhomogeneous Lipschitz space
$\blz_{\az}(\rn)$ for all $\az\in(0,\,1]$ (see \cite[Theorem 2]{NY85}).

To give the desired bilinear decomposition of the product $f\times g$, 
we first consider the boundedness
of the bilinear operators $\{\Pi_i\}_{i=1}^4$ on the products
of the Hardy and the Lipschitz spaces.

\begin{proposition}\label{pro 3.4}
Let $p\in(\frac{n}{n+1},\,1)$ and $\az=n(\frac{1}{p}-1)$. Then the bilinear operator $\Pi_1$,
defined as in \eqref{eqn 2.8}, is bounded
\begin{align}\label{eqn 3.5}
\text{from}\ H^p(\rn)\times\blz_{\az}(\rn) \  \text{to} \ H^p(\rn)
\end{align}
and
\begin{align}\label{eqn 3.6}
\text{from}\ H^p(\rn)\times\dot \blz_{\az}(\rn)\ \text{to} \ H^1(\rn).
\end{align}
\end{proposition}

\begin{proof}
We first prove \eqref{eqn 3.5}.  For any $g\in L^\fz(\rn)$, define the linear operator
$\Pi_g$ by setting, for any $f\in H^p(\rn)$,
$\Pi_g(f):=\Pi_1(f,\,g).$ Since $\Pi_1$ is a bilinear Calder\'on-Zygmund operator
with $\epsilon=1$ (see Proposition \ref{def 2.4}), we know that the kernel function
$K_g$ of the operator $\Pi_g$ satisfies that, for all $x$, $y\in \rn$ with $x\ne y$,
$\lf|K_g(x,\,y)\r|\ls \|g\|_{L^\fz(\rn)}\frac{1}{|x-y|^n}$
and, for all $x$, $x'$, $y\in\rn$ satisfying $x\ne y\ne x'$ and $|x-x'|\le |x-y|/2$,
\begin{align*}
\lf|K_g(x,\,y)-K_g(x',\,y)\r|+\lf|K_g(y,\,x)-K_g(y,\,x')\r|
\ls \|g\|_{L^\fz(\rn)}\frac{|x-x'|}{|x-y|^{n+1}}.
\end{align*}
This implies that $\Pi_g$
is a Calder\'on-Zygmund operator (see, for example, \cite[p.\,293]{St93}
for the precise definition of Calder\'on-Zygmund operators).

Moreover, for each $H^p(\rn)$-atom $a$, since $a$ has compact support, it follows
that
\begin{align*}
\dint_\rn \Pi_g(a)(x)\,dx&=\dint_\rn \Pi_1(a,\,g)(x)\,dx=\dsum_{\substack{I,\,I'\in \mathcal{D}\\
|I|=|I'|}} \dsum_{\lz\in E} \langle a,\,\phi_I \rangle
\langle g,\,\psi_{I'}^\lz \rangle\dint_\rn  \phi_I \psi_{I'}^\lz(x)\,dx= 0,
\end{align*}
which means that $\Pi_g^*(1)=0$. By using the
$T1$ theorem for Hardy spaces (see, for example,
\cite[Proposition 3.1]{yz08}), we know that,
for all $g\in L^\fz(\rn)$ and $f\in H^p(\rn)$ with $p\in(\frac{n}{n+1},\,1]$,
$\|\Pi_1(f,\,g)\|_{H^p(\rn)}=\|\Pi_g(f)\|_{H^p(\rn)}\ls \|g\|_{L^\fz(\rn)}\|f\|_{H^p(\rn)},$
which, together with the fact that $\blz_{\az}(\rn)\subset
L^\fz(\rn)$, implies that \eqref{eqn 3.5} holds true.

We now turn to the proof of \eqref{eqn 3.6}. Let $f\in H^p(\rn)$ and $g\in \dot \blz_{\az}(\rn)$.
Since the family $\{\psi_I^\lz\}_{I\in\mathcal{D},\,\lz\in E}$ of wavelets is an unconditional basis of
$H^p(\rn)$ (see, for example,
\cite[Theorem 5.2]{GM01}),
without loss of generality, we may assume that $f$ has a finite wavelet expansion.
In this case, by Theorem \ref{thm 3.3}, we know that $f=\sum_{l=1}^L \mu_l a_l$ has a
finite atomic decomposition with the same notation as therein.
Thus, by \eqref{eqn 2.6} and the fact that $\Pi_1$ is bounded from
$L^2(\rn)\times L^2(\rn)$ into $H^1(\rn)$ (see Theorem \ref{thm 2.2}(ii)), we find that
\begin{align}\label{eqn 3.x9}
\|\Pi_1(f,\,g)\|_{H^1(\rn)}&\ls \lf\{\dsum_{l=1}^L|\mu_l|^p\r\}^{\frac{1}
{p}}\|\Pi_1(a_l,\,b_l)\|_{H^1(\rn)}\\&\nonumber\ls
\|f\|_{H^p(\rn)}\|a_l\|_{L^2(\rn)}\|b_l\|_{L^2(\rn)},
\end{align}
where
\begin{align}\label{eqn 3.9}
b_l:=\dsum_{\substack{I\in\mathcal{D}
\\I\subset 2mR_l}}\dsum_{\lz\in E} \langle g,\,\psi_I^\lz \rangle \psi_I^\lz,
\end{align}
with $R_l$ being the dyadic cube related to $a_l$, satisfies that
\begin{align*}
\|b_l\|_{L^2(\rn)}\ls \dsup_{R\in \mathcal{D}}
\lf\{\dsum_{\substack{I\in\mathcal{D}
\\I\subset 2mR}}\dsum_{\lz\in E}\lf|\langle g,\,\psi_I^\lz \rangle\r|^2\r\}^{\frac{1}{2}}.
\end{align*}
By this, together with  \eqref{eqn 3.x9} and Theorems \ref{thm 3.2} and \ref{thm 3.3},
we conclude that
\begin{align}\label{eqn 3.x11}
\|\Pi_1(f,\,g)\|_{H^1(\rn)}&\ls
\|f\|_{H^p(\rn)}\dsup_{R\in \mathcal{D}}
\lf\{\frac{1}{|R|^{\frac{2}{p}-1}}\dsum_{\substack{I\in\mathcal{D}
\\I\subset 2mR}}\dsum_{\lz\in E}\lf|\langle g,\,\psi_I^\lz \rangle\r|^2\r\}^{\frac{1}{2}}\\
&\nonumber\ls \|f\|_{H^p(\rn)}\|g\|_{\dot \blz_{\az}(\rn)},
\end{align}
which immediately implies that \eqref{eqn 3.6} holds true and hence completes the proof of
Proposition \ref{pro 3.4}.
\end{proof}

We also have the following boundedness of the bilinear operators $\Pi_3$ and $\Pi_4$, whose
proofs are omitted due to their similarity to the proof of Proposition \ref{pro 3.4}.

\begin{proposition}\label{pro 3.6}
Let $p\in(\frac{n}{n+1},\,1)$ and $\az=n(\frac{1}{p}-1)$. Then the bilinear operator $\Pi_3$,
defined as in \eqref{eqn 2.10}, is bounded from $H^p(\rn)\times\blz_{\az}(\rn)$ to $H^p(\rn)$
and from $H^p(\rn)\times\dot \blz_{\az}(\rn)$ to $H^1(\rn)$.
\end{proposition}

\begin{proposition}\label{pro 3.5}
Let $p\in(\frac{n}{n+1},\,1)$ and $\az=n(\frac{1}{p}-1)$. Then the bilinear operator $\Pi_4$,
defined as in \eqref{eqn 2.11},
is bounded from $H^p(\rn)\times\dot \blz_{\az}(\rn)$ to $L^1(\rn)$.
\end{proposition}

We point out that we need to use
Theorem \ref{thm 2.2}(i), instead of Theorem \ref{thm 2.2}(ii),
in the proof of Proposition \ref{pro 3.5}, which justifies the space
$L^1(\rn)$ appearing in Proposition \ref{pro 3.5}.

We now establish the following boundedness of the bilinear operator $\Pi_2$.
Recall that, for any $p\in(0,\,\fz)$, non-negative weight function $w$
and measurable function $f$,
\begin{align}\label{3.12xxxx}
\|f\|_{L^p_w(\rn)}:=\lf[\dint_\rn |f(x)|^p\,w(x)\,dx\r]^{\frac{1}{p}}.
\end{align}

\begin{proposition}\label{pro 3.7}
Let $p\in(\frac{n}{n+1},\,1)$ and $\az=n(\frac{1}{p}-1)$. Then the bilinear operator $\Pi_2$,
defined as in \eqref{eqn 2.10}, is bounded
\begin{align}\label{eqn 3.14}
\text{from} \ H^p(\rn)\times\blz_{\az}(\rn) \ \text{to} \ H^p(\rn)
\end{align}
and
\begin{align}\label{eqn 3.15}
\text{from} \ H^p(\rn)\times \dot \blz_{\az}(\rn) \ \text{to} \ H^p_{w}(\rn),
\end{align}
where $w$ is as in \eqref{1.10x} and, for all $p\in(0,\,1]$, $H_w^p(\rn)$
denotes the {\it weighted Hardy space}
whose definition is similar to that of $H^p(\rn)$ in \eqref{eqn 3.xx2},
with the Lebesgue quasi-norm $\|\cdot\|_{L^p(\rn)}$ therein replaced by
the weighted Lebesgue quasi-norm
$\|\cdot\|_{L^p_w(\rn)}$ as in \eqref{3.12xxxx}.
\end{proposition}

\begin{proof}
The proof of \eqref{eqn 3.14} is similar to that of \eqref{eqn 3.5} in Proposition \ref{pro 3.4}, the
details being omitted. We only prove \eqref{eqn 3.15}.
To this end, let $f\in H^p(\rn)$ and $g\in \dot \blz_{\az}(\rn)$. Without loss of
generality, we may also assume that $f$ has a finite wavelet expansion. In this case, since $f$
has compact support and $g$ is continuous, we may restrict $g\in L^2(\rn)$ (see also
\eqref{eqn 3.x35} below for a similar discussion).  By Theorem \ref{thm 3.3},
we know that $f=\sum_{l=1}^L\mu_l a_l$ has a finite atomic decomposition as in Theorem
\ref{thm 3.3} with the notation same as therein. Thus,
\begin{align}\label{eqn 3.x18}
\Pi_2(f,\,g)=\sum_{l=1}^L \mu_l\Pi_2(a_l,\,g).
\end{align}

Let $\{P_j\}_{j\in\zz}$ and $\{Q_j\}_{j\in\zz}$ be the orthogonal projects as in \eqref{eqn 2.7}.
For each $l\in\{1,\,\ldots,\,L\}$,
let $R_l$ be the associated dyadic cube of the atom $a_l$ satisfying $|R_l|=2^{-j_l n}$ for some
$j_l\in\zz$. Since each $a_l$ also has a finite wavelet expansion, by an argument similar to that
used in the proof of \cite[Lemma 4.3]{BGK12}, we find that there exists $j_1\in\nn\cap (j_l,\,\fz)$
such that
\begin{align}\label{eqn 3.xx20}
\Pi_2(a_l,\,g)&=\dsum_{j=j_l}^{j_1-1} \lf(Q_j a_l\r)\lf(P_j g\r)=\dsum_{j=j_l}^{j_1-1}
Q_j a_l\lf[P_{j_l}g+\dsum_{i=j_l}^{j-1}Q_i g\r]\\
&\nonumber=a_l P_{j_l}g+\dsum_{j=j_l}^{j_1-1}
Q_j a_l\lf[\dsum_{i=j_l}^{j-1}Q_i g\r].
\end{align}

Using the fact that $g=P_{j_l}g+(I-P_{j_l})g$,
 \eqref{eqn 3.xx20} and the definitions of $P_j$
and $Q_j$, we know that
\begin{align*}\Pi_2(a_l,\,g)&=a_l P_{j_l}P_{j_l}g+\dsum_{j=j_l}^{j_1-1}
Q_j a_l\lf[\dsum_{i=j_l}^{j-1}Q_i \lf(I-P_{j_l}\r)g\r]=
a_l P_{j_l}g+\Pi_2(a_l,\,\lf[I-P_{j_l}\r]g).
\end{align*}

Moreover, by the fact that
\begin{align*}
\lf(I-P_{j_l}\r)g=\dsum_{j=j_l}^{\fz}
\dsum_{|I|=2^{-jn}}\dsum_{\lz\in E}\langle g,\,\psi_I^\lz\rangle \psi_I^\lz,
\end{align*}
$|R_l|=2^{-j_ln}$, \eqref{eqn 2.6} and \eqref{eqn 2.9}, we further conclude that
\begin{align*}
\Pi_2(a_l,\,g)&=a_l P_{j_l}g+\Pi_2\lf(a_l,\,\dsum_{I\subset 2m R_l}\dsum_{\lz\in E}
\langle g,\,\psi_I^\lz\rangle \psi_I^\lz\r)=
a_l P_{j_l}g+\Pi_2(a_l,\,b_l),
\end{align*}
where $b_l$ is as in \eqref{eqn 3.9}.

We now claim that there exists a positive constant $c$, independent of $g$, such that,
for any $l\in\{1,\,\ldots,\,L\}$, $\Pi_2(a_l,\,g)$ can be written as
\begin{align}\label{eqn 3.18}
\Pi_2(a_l,\,g)=h^{(1)}+ch^{(2)}g_{R_l},
\end{align}
where $g_{R_l}:=\frac{1}{|R_l|}\int_{R_l} g(x)\,dx$,
$h^{(1)}\in H^1(\rn)$ satisfies $\|h^{(1)}\|_{H^1(\rn)}\ls \|g\|_{\dot \blz_{\az}(\rn)}$
and $h^{(2)}$ is an $H^p(\rn)$-atom satisfying  $\|h^{(2)}\|_{H^p(\rn)}\ls 1$.

To show the above claim, recalling that $R_l$ is the associated dyadic cube of the atom $a_l$ and $|R_l|=2^{-j_l n}$,
we have
$a_l P_{j_l}g=\sum_{I\in\mathcal{D},\,|I|=2^{-j_l n}} a_l \langle g,\,\phi_I \rangle \phi_I$.
Observe that, for any $l\in\{1,\,\ldots,\,L\}$, $I\in\mathcal{D}$ and $|I|=2^{-j_l n}$,
\begin{align}\label{3.19y}
a_l\phi_I\ne 0\ \ \text{if and only if} \ \ I\subset 2m R_l
\end{align}
and
\begin{align}\label{3.19z}
\#\{I\in \mathcal{D}:\ |I|=2^{-j_ln},\, I\subset 2m R_l\}\le (2m)^n,
\end{align}
where $\# E$ for any set $E$ denotes the number of its elements and $m$ is as in \eqref{eqn 2.3}.
By \eqref{3.19y} and \eqref{3.19z}, we then know that the summation in $a_l P_{j_l}g$
is of finite terms with the number not more than $(2m)^n$.
Moreover, for each $I$, we write
\begin{align*}
a_l \langle g,\,\phi_I \rangle \phi_I&=a_l \phi_{I}|R_l|^{-\frac{1}{2}+\frac{1}{p}}
\langle g,\, |R_l|^{\frac{1}{2}-\frac{1}{p}}\phi_I \rangle\\
&=a_l\phi_{I}|R_l|^{-\frac{1}{2}+\frac{1}{p}}
\left\langle g,\, |R_l|^{1-\frac{1}{p}}\lf(\frac{1}{|R_l|^{1/2}}\phi_{I}-\frac{1}{|R_l|}\chi_{R_l}\r) \r\rangle
+a_l\phi_I |R_l|^{\frac{1}{2}} g_{R_l}.
\end{align*}
From \eqref{eqn 2.6}, $\supp a_l \subset R_l$ and $|R_l|=|I|$,
it follows that $\supp (a_l\phi_{I}|R_l|^{-\frac{1}{2}+\frac{1}{p}})\subset 2m R_l$.
Moreover, since $a_l$ is an $H^p(\rn)$-atom associated with $R_l$, we have
\begin{align*}
\lf\|a_l\phi_{I}|R_l|^{-\frac{1}{2}+\frac{1}{p}}\r\|_{L^2(\rn)}\le \lf|R_l\r|^{-\frac12}
\end{align*}
and, by Theorem \ref{thm 3.3}, we know that $a_l$ is of the form \eqref{eqn 3.x4}.
Thus, we conclude that
$\int_\rn a_l(x)\phi_{I}(x)|R_l|^{-\frac{1}{2}+\frac{1}{p}}\,dx=0,$
which implies that $a_l\phi_{I}|R_l|^{-\frac{1}{2}+\frac{1}{p}}$ is an $H^1(\rn)$-atom
associated with $2m R_l$.

By a similar calculation, we find that
$|R_l|^{1-\frac{1}{p}}(\frac{1}{|R_l|^{1/2}}\phi_{I}-\frac{1}{|R_l|}\chi_{R_l})$ is also
an $H^p(\rn)$-atom associated with $2mR_l$, which, together with the assumption
that $g\in \dot \blz_{\az}(\rn)$, shows that
\begin{align*}
\lf|\left\langle g,\, |R_l|^{1-\frac{1}{p}}\lf(\frac{1}{|R_l|^{1/2}}\phi_{I}-\frac{1}{|R_l|}\chi_{R_l}\r) \r\rangle\r|\ls
\|g\|_{\dot \blz_{\az}(\rn)}.
\end{align*}
Thus, we conclude that
\begin{align}\label{eqn 3.19}
\hs\hs\hs\lf\|a_l\phi_{I}|R_l|^{-\frac{1}{2}+\frac{1}{p}}
\left\langle g,\, |R_l|^{1-\frac{1}{p}}\lf(\frac{1}
{|R_l|^{1/2}}\phi_{I}-\frac{1}{|R_l|}\chi_{R_l}\r) \r\rangle\r\|_{H^1(\rn)}\ls \|g\|_{\dot \blz_{\az}(\rn)}.
\end{align}
Now, let
\begin{align*}
h^{(1)}:=\Pi_2(a_l,\,b_l)+\dsum_{I\in\mathcal{D},\,
|I|=2^{-j_l n}}a_l\phi_{I}|R_l|^{-\frac{1}{2}+\frac{1}{p}}
\left\langle g,\, |R_l|^{1-\frac{1}{p}}\lf(\frac{1}{|R_l|^{1/2}}\phi_{I}-\frac{1}{|R_l|}\chi_{R_l}\r) \r\rangle.
\end{align*}
By the fact that $\Pi_2$ is bounded from $L^2(\rn)\times L^2(\rn)$ to $H^1(\rn)$
(see Theorem \ref{thm 2.2}(ii)) and a calculation
similar to \eqref{eqn 3.x11}, we know that
$$\lf\|\Pi_2(a_l,\,b_l)\r\|_{H^1(\rn)}\ls \|a_l\|_{L^2(\rn)}\|b_l\|_{L^2(\rn)}\ls \|g\|_{\dot \blz_{\az}(\rn)},$$
which, combined with \eqref{eqn 3.19}, implies that $h^{(1)}\in H^1(\rn)$ and
\begin{align}\label{eqn 3.21}
\lf\|h^{(1)}\r\|_{H^1(\rn)}\ls \|g\|_{\dot \blz_{\az}(\rn)}.
\end{align}
Let $h^{(2)}:=\sum_{I\in\mathcal{D},\,
|I|=2^{-j_l n}}\frac{1}{c}a_l\phi_I |R_l|^{\frac{1}{2}}$, where $c$ is a positive constant such that
\begin{align}\label{3.24y}
\lf\|\frac{1}{c}\dsum_{|I|=2^{j_l n}}a_l\phi_I
|R_l|^{\frac{1}{2}}\r\|_{L^2(\rn)}\le |2mR_l|^{\frac{1}{2}-\frac{1}{p}}.
\end{align}
Using \eqref{3.19y} and \eqref{3.19z}, we choose a positive constant $c$,
depending only on $m$ from \eqref{eqn 2.3}, such that  \eqref{3.24y} holds true.
This, together with the fact that $a_l$ is an $H^p(\rn)$-atom associated with $R_l$ of the form
\eqref{eqn 3.x4} and the assumption $p\in(\frac{n}{n+1},\,1)$, implies that
$h^{(2)}$ is an $H^p(\rn)$-atom associated with $2mR_l$. Combining this and \eqref{eqn 3.21},
we find that the claim \eqref{eqn 3.18} holds true.

With the help of the claim \eqref{eqn 3.18}, we now continue the proof of \eqref{eqn 3.15}.
That is, we need to prove that both $h^{(1)}$ and $ch^{(2)}g_{R_l}\in H_w^p(\rn)$.
From \eqref{1.10x} and H\"older's inequality,
we deduce that, for any $h\in H^1(\rn)$,
\begin{align*}
\|h\|_{H_{w}^p(\rn)}&=\|h^*\|_{L_{w}^p(\rn)}=
\lf\{\dint_\rn \frac{\lf|h^*(x)\r|^p}{(1+|x|)^{n(1-p)}}\,dx\r\}^{\frac{1}{p}}\\
&\ls \dsum_{j=0}^\fz  \dint_{S_j(B_0)}\lf|h^*(x)\r|\,dx
\sim \|h^*\|_{L^1(\rn)}\sim\|h\|_{H^1(\rn)},
\end{align*}
where $h^*$ denotes the non-tangential grand maximal function defined
as in \eqref{eqn 3.xx2}, $B_0$ is the unit ball centered at $0$ and, for any $j\in\nn$,
$S_j(B_0):=(2^{j+1}B_0)\setminus (2^{j}B_0)$ and $S_0(B_0):=
2B_0$. This shows that $H^1(\rn)\subset H_{w}^p(\rn)$ and hence
\begin{align}\label{eqn 3.22}
\lf\|h^{(1)}\r\|_{H_w^p(\rn)}\ls \|g\|_{\dot \blz_{\az}(\rn)}.
\end{align}

To estimate $\|ch^{(2)}g_{R_l}\|_{H_{w}^p(\rn)}$, we first write
\begin{align}\label{eqn 3.23}
\lf(ch^{(2)}g_{R_l}\r)^*\ls \lf(h^{(2)}\r)^*g_{R_l}\ls \lf(h^{(2)}\r)^*|g|+\lf(h^{(2)}\r)^*\lf|g-g_{R_l}\r|.
\end{align}
From the fact that $h^{(2)}$ is an $H^p(\rn)$-atom and the definition of
$\dot \blz_{\az}(\rn)$, we deduce that
\begin{align}\label{eqn 3.25}
&\lf\|\lf(h^{(2)}\r)^*|g|\r\|_{L^p_{w}(\rn)}\\
&\nonumber\hs\ls \lf\{\dint_{\rn}
\frac{|g(x)-g(0)|^p}{(1+|x|)^{n(1-p)}}\lf[\lf(h^{(2)}\r)^*(x)\r]^p\,dx\r\}^{\frac{1}{p}}
+|g(0)|\lf\|h^{(2)}\r\|_{H^p(\rn)}\\ \nonumber&\hs\ls \|g\|_{\dot \blz_{\az}(\rn)}+|g(0)|
\sim\|g\|_{\dot \blz_{\az}^+(\rn)},
\end{align}
where $\|g\|_{\dot \blz_{\az}^+(\rn)}:=\|g\|_{\dot \blz_{\az}(\rn)}+|g(0)|$.

We now estimate the term $\|(h^{(2)})^* |g-g_{R_l} |\|_{L^p_{w}(\rn)}$.
Recall that $h^{(2)}$ is an $H^p(\rn)$-atom associated with $2mR_l$. To simplify the
calculations, without loss of generality,
we may assume that $R_l:=\mathbb{Q}_0$ is the unit cube centered at $0$ in the remainder of this proof.
It is known (see, for example, \cite[p.\,106]{St93}) that, for all $x\in(4m\mathbb{Q}_0)^\complement$,
\begin{align*}
\lf(h^{(2)}\r)^*(x)\ls \frac{1}{(1+|x|)^{n+1}}.
\end{align*}
This, together with the definition of $\dot \blz_{\az}(\rn)$ and the assumption
that $p\in(\frac{n}{n+1},\,1]$, shows that
\begin{align}\label{eqn 3.30}
&\lf\|\lf(h^{(2)}\r)^*\lf|g-g_{\mathbb{Q}_0}\r|\r\|_{L^p_{w}(\rn)}
\ls\|g\|_{\dot \blz_{\az}(\rn)}\lf\{\dint_{\rn}\frac{1}
{(1+|x|)^{(n+1)p}}\,dx\r\}^{\frac{1}{p}}\ls \|g\|_{\dot \blz_{\az}(\rn)}.
\end{align}

Combining \eqref{eqn 3.23}, \eqref{eqn 3.25} and \eqref{eqn 3.30}, we conclude that
$\|ch^{(2)}g_{R_l}\|_{H_{w}^p(\rn)}\ls \|g\|_{\dot \blz^+_{\az}(\rn)},$
where $\|\cdot\|_{\dot \blz^+_{\az}(\rn)}$ is as in \eqref{eqn 3.25},
which, together with \eqref{eqn 3.x18}, \eqref{eqn 3.18} and \eqref{eqn 3.22}, implies that
\begin{align*}
\lf\|\Pi_2(f,\,g)\r\|_{H^p_{w}(\rn)}\ls \lf(\dsum_{l=1}^L |\mu_l|^p\r)^{\frac{1}{p}}
\lf\|\Pi_2(a_l,\,g)\r\|_{H_{w}^p(\rn)}\ls \|f\|_{H^p(\rn)} \|g\|_{\dot \blz^+_{\az}(\rn)}.
\end{align*}
This immediately shows that \eqref{eqn 3.15} holds true and hence finishes the proof of
Proposition \ref{pro 3.7}.
\end{proof}

We now state the main result of this section.

\begin{theorem}\label{thm 3.8}
Let $p\in(\frac{n}{n+1},\,1)$ and $\az:=n(\frac{1}{p}-1)$.
\begin{itemize}
\item[{\rm (i)}]
It holds true that there exist two bounded bilinear operators
$S: H^p(\rn)\times \dot \blz_{\az}(\rn)\to L^1(\rn)$ and
$T: H^p(\rn)\times \dot \blz_{\az}(\rn)\to H_w^p(\rn)$, with $w$ as in \eqref{1.10x},
such that, for all $(f,g)\in H^p(\rn)\times \dot \blz_{\az}(\rn)$,
\begin{align*}
f\times g=S(f,\,g)+T(f,\,g).
\end{align*}

\item[{\rm (ii)}]
It holds true that there exist two bounded bilinear operators
$S: H^p(\rn)\times  \blz_{\az}(\rn)\to L^1(\rn)$ and
$T: H^p(\rn)\times \blz_{\az}(\rn)\to H^p(\rn)$
such that, for all $(f,g)\in H^p(\rn)\times \blz_{\az}(\rn)$,
\begin{align*}
f\times g=S(f,\,g)+T(f,\,g).
\end{align*}
\end{itemize}
\end{theorem}

\begin{proof}
We first prove (i). For all $f\in H^p(\rn)$ and $g\in \dot \blz_{\az}(\rn)$, let $\{f_k\}_{k\in\nn}\subset
H^p(\rn)$ have finite wavelet expansions and satisfy
$\lim_{k\to \fz} f_k=f$ in $H^p(\rn)$. By the definition of $f\times g$ in \eqref{eqn 3.3}, we know that
\begin{align}\label{eqn 3.31}
f\times g=\dlim_{k\to \fz} f_k \,g \ \ \ \ \text{in}\ \mathcal{S}'(\rn),
\end{align}
where, for each $k\in\nn$, $f_k\, g$ denotes the usual pointwise multiplication of $f_k$
and $g$. Since $f_k$ has a finite wavelet expansion, we know that $f_k\in L^2(\rn)$ and
has compact support. Let $\eta_k$ be a cut-off function satisfying $\supp \eta_k\subset
8m (\supp f_k)$ and $\eta_k\equiv 1$ on $4m (\supp f_k)$, where $m$ is as in
\eqref{eqn 2.3}. It is easy to see that $f_k \,g=f_k \,(\eta_kg)$ and $\eta_k g\in L^2(\rn)$.
From \eqref{eqn 2.7}, we deduce that
$f_k \,g=f_k (\eta_k g)=\sum_{i=1}^4\Pi_i(f_k,\,\eta_kg)$
in the sense of $L^1(\rn)$ and hence of distributions.

By \eqref{eqn 2.8}, \eqref{eqn 2.6} and the definition of $\eta_k$,
we find that
\begin{align}\label{eqn 3.x35}
\Pi_1(f_k,\,\eta_kg)&=\dsum_{\substack{I,\,I'\in \mathcal{D}\\
|I|=|I'|}} \dsum_{\lz\in E} \langle f_k,\,\phi_I \rangle
\langle \eta_kg,\,\psi_{I'}^\lz \rangle \phi_I \psi_{I'}^\lz\\
&\nonumber=\dsum_{\substack{I,\,I'\in \mathcal{D}\\
|I|=|I'|}} \dsum_{\lz\in E} \langle f_k,\,\phi_I \rangle
\langle g,\,\psi_{I'}^\lz \rangle \phi_I \psi_{I'}^\lz=\Pi_1(f_k,\,g).
\end{align}
Similarly, we obtain $\Pi_i(f_k,\,\eta_kg)=\Pi_i(f_k,\,g)$ for
$i\in\{2,\,3,\,4\}$. Thus, we find that
$f_k \,g=\sum_{i=1}^4\Pi_i(f_k,\,g)$
holds true in the sense of distributions. Combining \eqref{eqn 3.31} and
Propositions \ref{pro 3.4} through \ref{pro 3.7}, we find that
\begin{align*}
f\times g&=\dsum_{i=1}^4\dlim_{k\to \fz} \Pi_i(f_k,\,g)
=\Pi_4(f,\,g)+\lf[\dsum_{i=1}^3\Pi_i(f,\,g)\r]
=:S(f,\,g)+T(f,\,g),
\end{align*}
where $S(f,\,g):=\Pi_4(f,\,g)\in L^1(\rn)$ and
\begin{align*}
T(f,\,g):=&\sum^3_{i=1}\Pi_i(f,\,g)\in
H^1(\rn)+H_w^p(\rn)\subset H_w^p(\rn).
\end{align*}
This shows that (i) holds true.

The proof of (ii) is similar to that of (i). We only need to let
$S(f,\,g):=\Pi_4(f,\,g)\in L^1(\rn)$ and
$T(f,\,g):=\sum_{i=1}^3\Pi_i(f,\,g) \in H^p(\rn),$
which shows that (ii) holds true and hence completes the proof of Theorem \ref{thm 3.8}.
\end{proof}

\begin{remark}\label{rem 3.9}
(i) Recall that Bonami et al. \cite{BF10} proved that, for all $p\in(0,\,1)$ and
$\az:=n(\frac{1}{p}-1)$,
\begin{align*}
H^p(\rn)\times \dot\blz_{\az}(\rn)\subset L^1(\rn)+H_{w_\gz}^p(\rn)
\end{align*}
with $\gz\in(n(1-p),\,\fz)$ and $w_\gz$ as in \eqref{eqn 1.4x}. Theorem
\ref{thm 3.8} improves the above result in two aspects when $p\in (\frac{n}{n+1},\,1)$.
First, for all $\gz\in(n(1-p),\,\fz)$, since, for all $x\in\rn$,
$$w_\gz(x):=\frac{1}{(1+|x|)^\gz}<\frac{1}{(1+|x|)^{n(1-p)}}=:w(x),$$
it follows that Theorem \ref{thm 3.8} obtains a range space $L^1(\rn)+H_{w}^p(\rn)$ which
is smaller than the range space $L^1(\rn)+H_{w_\gz}^p(\rn)$ obtained in \cite{BF10}.
Second, the decomposition obtained in Theorem \ref{thm 3.8} is bilinear.
Recall that, in the case $p=1$, whether or not the product space $H^1(\rn)\times\bbmo(\rn)$
has a bilinear decomposition consists one of the conjectures proposed in \cite{BIJZ07},
which was finally solved by Bonami et al. \cite{BGK12}.

(ii) We also point out that the bilinear decomposition obtained in Theorem \ref{thm 3.8}(i) is
sharp in the sense that the range space $L^1(\rn)+H_w^p(\rn)$ is
almost smallest. To explain this, let $A$ be a space smaller than $H_w^p(\rn)$ and satisfy
the decomposition as follows:
$$H^p(\rn)\times \dot\blz_{\az}(\rn)\subset L^1(\rn)+A.$$
Then we have
\begin{align}\label{eqn 3.e43}
\lf[H_w^p(\rn))^*\cap L^\fz(\rn\r]&= \lf(L^1(\rn)+H_w^p(\rn)\r)^*
\subset \lf(L^1(\rn)+A\r)^*\\
&= \lf[A^*\cap L^\fz(\rn)\r].\nonumber
\end{align}

Moreover, For all $p\in(\frac{n}{n+1},\,1)$, $r\in(0,\,\fz)$ and $x\in\rn$, let
$\psi(x,\,r):=[\frac{r}{1+|x|}]^{n(\frac{1}{p}-1)}$ and $w$ be as in \eqref{1.10x}.
Define the {BMO type space} $\bbmo_{\psi}(\rn)$
by setting
\begin{align*}
\bbmo_{\psi}(\rn)&:=\Bigg\{f\in L_\loc^1(\rn):\ \lf\|f\r\|_{\bbmo_{\psi}(\rn)}\\
&\nonumber\hspace{2cm}\hspace{-0.1cm}\lf.:=
\dsup_{\substack{x\in\rn\\r\in(0,\,\fz)}}
\lf[\frac{1}{\psi(x,\,r)|B(x,\,r)|}{\dint_{B(x,\,r)}\lf|f(y)-f_{B(x,\,r)}\r|\,dy}\r]<\fz\r\},
\end{align*}
where $f_{B(x,\,r)}:=\frac{1}{|B(x,\,r)|}\int_{B(x,\,r)}f(y)\,dy$.
By \cite[Theorem 3.2]{Ky14} and some elementary technical calculations, we know that
$$\lf[(H_w^p(\rn))^*\cap L^\fz(\rn)\r]=\lf[\bbmo_\psi(\rn)\cap L^\fz(\rn)\r],$$
which coincides with the pointwise multiplier class of $\dot \blz_{n(\frac{1}{p}-1)}(\rn)$
(see \cite[Theorem 1]{NY85}). By the fact that the product of $H^p(\rn)\times \dot \blz_\az(\rn)$
is well defined for all pointwise multipliers of $\dot \blz_{n(\frac{1}{p}-1)}(\rn)$ as in \eqref{eqn 3.4},
if the duality argument is possible (which may not be the case since
the function spaces dealt here may not be Banach spaces), we conclude that
\begin{align*}
\lf[A^*\cap L^\fz(\rn)\r]\subset \lf[\bbmo_\psi(\rn)\cap L^\fz(\rn)\r]=
\lf[(H_w^p(\rn))^*\cap L^\fz(\rn)\r],
\end{align*}
which, together with \eqref{eqn 3.e43}, implies that
$\lf[A^*\cap L^\fz(\rn)\r]=\lf[(H_w^p(\rn))^*\cap L^\fz(\rn)\r].$
In this sense, we say that $H_w^p(\rn)$ and $A$ coincide with each other and hence the range space
$L^1(\rn)+H_w^p(\rn)$ is almost smallest.

We point out that the bilinear decomposition of Theorem \ref{thm 3.8}(ii) is also
almost smallest in the same meaning as in Theorem \ref{thm 3.8}(i),
in view of the fact that the pointwise multiplier class of $\blz_{n(\frac{1}{p}-1)}(\rn)$
coincides with the space $\blz_{n(\frac{1}{p}-1)}(\rn)\cap L^\fz(\rn)$ which is just
$(L^1(\rn)+H^p(\rn))^*$
(see \cite[Theorem 3]{NY85}), the details being omitted.
\end{remark}

\section{Applications to the Div-Curl Lemma}\label{s5}

\hskip\parindent   In this section, we apply the bilinear
decompositions of the products of Hardy spaces and their duals,
obtained in Section \ref{s3},
to the study of the div-curl lemma at the endpoint case $q=\fz$. To this end,
we first introduce some notation. For all $p\in(0,\,\fz)$, let
\begin{align}\label{eqn 5.1}
H^p(\rn;\,\rn):=\{\mathbf{F}:=(F_1,\,\ldots,\,F_n):\ \ \text{for all}\ i\in\{1,\,\ldots,\,n\}, \ F_i\in H^p(\rn)\}
\end{align}
and, for any $\mathbf{F}\in H^p(\rn;\,\rn)$, let
\begin{align*}
\|\mathbf{F}\|_{H^p(\rn;\,\rn)}:=\lf[\dsum_{i=1}^n\|F_i\|_{H^p(\rn)}^2\r]^{\frac{1}{2}}.
\end{align*}
The vector-valued spaces $h^p(\rn;\,\rn)$,
$\dot\blz_\az(\rn;\,\rn)$ and $\bmo\,(\rn;\,\rn)$ are defined similarly, the details being omitted.

The following theorem is the main result of this section.

\begin{theorem}\label{thm 5.1}
For all $p\in(\frac{n}{n+1},\,1)$ and $\az:=n(\frac{1}{p}-1)$, let $\mathbf{F}\in H^p(\rn;\,\rn)$
with ${\rm curl} \,\mathbf{F}\equiv0$ and $\mathbf{G}\in \dot\blz_{\az}(\rn;\,\rn)$
with ${\rm div}\, \mathbf{G}\equiv0$ (both of the equalities hold true in the sense of distributions).
Then $\mathbf{F}\cdot
\mathbf{G}\in H^p_{w}(\rn)$,
where $w$ is as in \eqref{1.10x}.
\end{theorem}

\begin{proof}
Let $\mathbf{F}\in H^p(\rn;\,\rn)$
with ${\rm curl} \,\mathbf{F}\equiv0$ and $\mathbf{G}\in \dot\blz_{\az}(\rn;\,\rn)$
with ${\rm div}\, \mathbf{G}\equiv0$. By Theorem \ref{thm 3.8}, we know that
\begin{align*}
\mathbf{F}\cdot \mathbf{G}=\dsum_{i=1}^n F_i\times G_i= \dsum_{i=1}^n S(F_i,\,G_i)+ \dsum_{i=1}^nT(F_i,\,G_i)
=:\mathrm{A}(\mathbf{F},\,\mathbf{G})+\mathrm{B}(\mathbf{F},\,\mathbf{G}).
\end{align*}

From Theorem \ref{thm 3.8}(i), it immediately follows that $\mathrm{B}
(\mathbf{F},\,\mathbf{G})\in H^p_w(\rn)$ and
\begin{align*}
\lf\|\mathrm{B}(\mathbf{F},\,\mathbf{G})\r\|_{H^p_w(\rn)}\ls \|\mathbf{F}\|_{H^p(\rn;\,\rn)}
\|\mathbf{G}\|_{ \dot\blz_{\az}(\rn;\,\rn)}.
\end{align*}
Thus, to finish the proof of (i),
it suffices to show $\mathrm{A}(\mathbf{F},\,\mathbf{G})\in H^p_w(\rn)$. To this end,
we only need to prove that $\mathrm{A}(\mathbf{F},\,\mathbf{G})\in H^1(\rn)$ because of the inclusion
$H^1(\rn)\subset H^p_w(\rn)$. Since $L^2(\rn;\,\rn)$ is dense in $H^p(\rn;\,\rn)$,
without loss of generality, we may assume that
$\mathbf{F}\in H^p(\rn;\,\rn)\cap L^2(\rn;\,\rn)$. Using the Helmholtz decomposition
(see, for example, \cite[Section 4]{BIJZ07} for more details), we find that there exists
$$f:=-\sum_{i=1}^n R_i(F_i)\in H^p(\rn)\cap L^2(\rn)$$ such that
$\mathbf{F}=\nabla (-\Delta)^{-1/2} f$, where, for all $i\in\{1,\,\ldots,\,n\}$,
$R_i$ denotes the {\it $j$-th Riesz transform}. Moreover, since ${\rm div}\, \mathbf{G}\equiv 0$,
it follows that $\sum_{i=1}^n R_i (G_i)\equiv 0$. Thus, we can write
\begin{align*}
\mathrm{A}(\mathbf{F},\,\mathbf{G})=\dsum_{i=1}^n S(F_i,\,G_i)=\dsum_{i=1}^n
\lf[S(R_i(f),\,G_i)+S(f,\,R_i(G_i))\r].
\end{align*}
Using \eqref{eqn 2.11} and the fact that $R_i$ is a Calder\'on-Zygmund operator with odd kernel,
we further find that, for each $i\in\{1,\,\ldots,\,n\}$,
\begin{align*}
&S(R_i(f),\,G_i)+S(f,\,R_i(G_i))\\
&\hs=\dsum_{I,\,I'\in\mathcal{D}}\dsum_{\lz,\,\lz'\in E}
\langle f,\,\psi_I^\lz \rangle \langle G_i,\,\psi_{I'}^{\lz'}
\rangle \langle R_i \psi_I^\lz,\,\psi_{I'}^{\lz'} \rangle
\lf(\psi_{I'}^{\lz'}\r)^2\\
&\hs\hs+\dsum_{I,\,I'\in\mathcal{D}}\dsum_{\lz,\,\lz'\in E}
\langle f,\,\psi_I^\lz \rangle \langle G_i,\,\psi_{I'}^{\lz'} \rangle \langle  \psi_I^\lz,\,R_i\psi_{I'}^{\lz'} \rangle
\lf(\psi_{I}^{\lz}\r)^2\\
&\hs=\dsum_{I,\,I'\in\mathcal{D}}\dsum_{\lz,\,\lz'\in E}
\langle f,\,\psi_I^\lz \rangle \langle G_i,\,\psi_{I'}^{\lz'} \rangle
\langle R_i \psi_I^\lz,\,\psi_{I'}^{\lz'} \rangle
\lf[\lf(\psi_{I'}^{\lz'}\r)^2-\lf(\psi_{I}^{\lz}\r)^2\r],
\end{align*}
which, together with some calculations similar to those used in the proof of \cite[Lemma 6.1]{BGK12},
implies that
\begin{align*}
\lf\|S(R_i(f),\,G_i)+S(f,\,R_i(G_i))\r\|_{H^1(\rn)}\ls
\dsum_{I,\,I'\in\mathcal{D}}\dsum_{\lz,\,\lz'\in E} \lf|\langle f,\,\psi_I^\lz \rangle\r|
\lf| \langle G_i,\,\psi_{I'}^{\lz'} \rangle\r|p_\dz(I,\,I'),
\end{align*}
where, for all $\dz\in(0,\,\frac{1}{2}]$, $|I|=2^{-jn}$ and $|I'|=2^{-j'n}$ with centers
at $x_I$, respectively, $x_{I'}$,
\begin{align*}
p_\dz(I,\,I'):=2^{-|j-j'|(\dz+n/2)}\lf(\frac{2^{-j}+2^{-j'}}{2^{-j}+2^{-j'}+|x_I-x_{I'}|}\r)^{n+\dz}.
\end{align*}
This shows that the coefficient matrix of $\mathrm{A}$ is almost diagonal (see \cite[Theorem 3.3]{FJ90}
for more details), which, combined with a dual argument on the sequence space, the boundedness of the
Riesz transform $\nabla(-\Delta)^{-1/2}$ on $H^p(\rn)$ (see, for example, \cite[p.\,115, Theorem 4]{St93})
and Theorem \ref{thm 3.2}, implies that
\begin{align*}
&\lf\|S(R_i(f),\,G_i)+S(f,\,R_i(G_i))\r\|_{H^1(\rn)}\\
&\hs\ls \dsum_{I'\in\mathcal{D}}\dsum_{\lz'\in E}\lf[ \dsum_{I\in\mathcal{D}}
\dsum_{\lz\in E}\lf| \langle f,\,\psi_{I}^{\lz} \rangle\r|p_\dz(I,\,I')\r]
\lf|\langle G_i,\,\psi_{I'}^{\lz'} \rangle\r|\\
&\hs\ls \lf\|\lf\{\dsum_{I\in\mathcal{D}}\dsum_{\lz\in E}\lf| \langle f,\,\psi_{I}^{\lz}
\rangle\r|p_\dz(I,\,I')\r\}_{I'\in \mathcal{D},\,\lz'\in E}\r\|_{\dot f^0_{p,\,2}(\rn)}
\lf\|\lf\{\langle G_i,\,\psi_{I'}^{\lz'} \rangle\r\}_{I'\in \mathcal{D},\,\lz'\in E}\r\|_{\mathcal{C}_{p}(\rn)}\\
&\hs\ls \lf\|\lf\{ \langle f,\,\psi_{I}^{\lz} \rangle\r\}_{I\in \mathcal{D},\,\lz\in E}\r\|_{\dot f^0_{p,\,2}(\rn)}
\lf\|\lf\{\langle G_i,\,\psi_{I'}^{\lz'} \rangle\r\}_{I'\in \mathcal{D},\,\lz'\in E}\r\|_{\mathcal{C}_{p}(\rn)}\\
&\hs\ls\|\mathbf{F}\|_{H^p(\rn;\,\rn)}\|\mathbf{G}\|_{\dot\blz_\az(\rn;\,\rn)},
\end{align*}
where
$\lf\|\cdot\r\|_{\dot{f}^{0}_{p,\,2}(\rn)}$ and
$\|\cdot\|_{\mathcal{C}_{p}(\rn)}$ is defined as in \eqref{3.3xx}, respectively,
Theorem \ref{thm 3.2}. This shows that $\mathrm{A}(\mathbf{F},\,\mathbf{G})\in H^1(\rn)$ and
$$\lf\|\mathrm{A}(\mathbf{F},\,\mathbf{G})\r\|_{H^1(\rn)}\ls
\|\mathbf{F}\|_{H^p(\rn;\,\rn)}\|\mathbf{G}\|_{\dot\blz_\az(\rn;\,\rn)},$$
which completes the proof of Theorem \ref{thm 5.1}.
\end{proof}

{\bf Acknowledgements.} This article was completed when the second author was visiting to Vietnam Institute
for Advanced Study in Mathematics (VIASM), who would like to thank the VIASM for financial support and hospitality.


\bigskip

{\small\noindent Jun Cao

\medskip

\noindent
Department of Applied Mathematics, Zhejiang University of Technology,
Hangzhou 310023, People's Republic of China

\smallskip

\noindent{\it E-mail:} \texttt{caojun1860@zjut.edu.cn}

\bigskip

\noindent Luong Dang Ky

\medskip

\noindent Department of Mathematics, University of Quy Nhon,
170 An Duong Vuong, Quy Nhon, Binh Dinh, Viet Nam

\smallskip

\noindent{\it E-mail:} \texttt{dangky@math.cnrs.fr}

\bigskip

\noindent Dachun Yang (Corresponding author)

\medskip

\noindent School of Mathematical Sciences, Beijing Normal
University, Laboratory of Mathematics and Complex Systems, Ministry
of Education, Beijing 100875, People's Republic of China

\smallskip

\noindent{\it E-mail:} \texttt{dcyang@bnu.edu.cn}}

\end{document}